\documentclass{amsart}
\usepackage{amsfonts}
\usepackage{graphicx}
\usepackage{amscd}
\usepackage{amsmath}
\usepackage{amssymb}
\theoremstyle{plain}

\newtheorem{corollary}{\bf Corollary}
\newtheorem{definition}{\bf Definition}
\newtheorem{example}{\bf Example}
\newtheorem{lemma}{\bf Lemma}

\newtheorem{theorem}{\bf Theorem}
\numberwithin{equation}{section}
\begin{document}
\title[generalized $m$-quasi-Einstein metrics]{Characterizations and integral formulae for generalized $m$-quasi-Einstein metrics}
\author{Abd\^{e}nago Barros$^{1}$}
\address{$^{1}$ Departamento de Matem\'{a}tica-UFC\\
60455-760-Fortaleza-CE-BR} \email{abbarros@mat.ufc.br}
\thanks{$^{1,2}$ Partially supported by CNPq-BR}
\author{Ernani Ribeiro Jr$^{2}$}
\address{$^{2}$ Departamento de Matem\'{a}tica-UFC\\
60455-760-Fortaleza-CE-BR} \email{ernani@mat.ufc.br}
%\thanks{$^{2}$ Partially supported by CNPq-BR}
\urladdr{http://www.mat.ufc.br/pgmat}
\date{May 29, 2012}
\begin{abstract}
The aim of this paper is to present some structural equations for generalized $m$-quasi-Einstein metrics $(M^{n},\,g,\,\nabla f,\,\lambda),$ which was defined recently by Catino in \cite{catino}. In addition, supposing that $M^n$ is an Einstein manifold we shall show that it is a space form with a well defined potential $f.$ Finally, we shall derive a formula for the Laplacian of its scalar curvature which will give some integral formulae for such a class of compact manifolds that permit to obtain some rigidity results.
\end{abstract}

\maketitle

\section{Introduction and statement of the main results}
In recent years, much attention has been given to classification of Riemannian manifolds admitting an Einstein-like structure, which are natural generalization of the classical Ricci solitons. For instance, Catino in \cite{catino} introduced a class of special Riemannian metrics which naturally generalizes the Einstein condition. More precisely, he defined that a complete Riemannian manifold $(M^{n},g)\,,n\ge2$, is a \emph{generalized quasi-Einstein metric} if there exist three smooth functions $f$, $\lambda$ and $\mu$ on $M$, such that
\begin{equation}
\label{eqCat}
Ric+\nabla^{2}f-\mu df\otimes df=\lambda g,
\end{equation}where $Ric$ denotes the Ricci tensor of $(M^n,\,g),$ while $\nabla^{2}$  and $\otimes $ stand for the Hessian and the tensorial product, respectively.

As a particular case of (\ref{eqCat}) we shall consider the following.
\begin{definition}
We say that $(M^n,\,g)$ is a \emph{generalized $m$-quasi-Einstein metric} if there exist two smooth functions $f$ and $\lambda$ on $M$ satisfying
\begin{equation}
\label{eqprinc}
Ric+\nabla^{2}f-\frac{1}{m}df\otimes df=\lambda g,
\end{equation}
where  $0<m\leq\infty$ is an integer. The tensor $Ric_{f}=Ric+\nabla^{2}f-\frac{1}{m}df\otimes df$ is called Bakry-Emery Ricci tensor.
\end{definition}

In particular, we have

\begin{equation}
\label{eqprinc1}
{Ric}(\nabla f,\nabla f)+\langle \nabla_{\nabla f}\nabla f,\nabla f\rangle =\frac{1}{m}|\nabla f|^{4}+\lambda |\nabla f|^{2},
\end{equation}where $\langle\,,\,\rangle$ and $|\,\,|$ stand for the metric $g$ and its associated norm, respectively.

Moreover, if $R$ stands for the scalar curvature of $(M^n,\,g)$, then, taking trace of both members of equation (\ref{eqprinc}) we deduce
\begin{equation}
\label{eqntrace}
R+\Delta f-\frac{1}{m}|\nabla f|^{2}=\lambda n.
\end{equation}

Thereby we derive
\begin{equation}
\label{eqntrace2}
\langle \nabla f,\nabla R\rangle+\langle \nabla f,\nabla \Delta f\rangle=\frac{1}{m}\langle\nabla f, \nabla|\nabla f|^{2}\rangle+n\langle\nabla \lambda,\nabla f\rangle.
\end{equation}

One notices that combining equations (\ref{eqprinc}) and (\ref{eqntrace}) we infer
\begin{equation}
\label{eqntrfr}
{\nabla^{2}f-\frac{\Delta f}{n}\,g=\frac{1}{m}\big(df\otimes df-\frac{1}{n}|\nabla f|^2g\big)-\big(Ric-\frac{R}{n}g\big)}.
\end{equation}

It is important to point out that if $m=\infty$ and $\lambda$ is constant,  equation (\ref{eqprinc}) reduces to one associated to a gradient Ricci soliton, for a good survey in this subject we recommend the work due to Cao in \cite{cao}, as well as if $\lambda$ is only constant and $m$ is a positive integer, it corresponds to $m$-quasi-Einstein metrics that are exactly those $n$-dimensional manifolds which are the base of an $(n+m)$-dimensional Einstein warped product, for more details see \cite{csw}, \cite{case}, \cite{hepeterwylie} and \cite{brG}. The $1$-quasi-Einstein metrics satisfying $\Delta e^{-f}+\lambda e^{-f}=0$ are more commonly called \emph{static metrics}, for more details see \cite{Corvino}. Static metrics have been studied extensively for their connection to scalar curvature, the positive mass theorem and general relativity, see e.g. \cite{Anderson}, \cite{AndersonKhuri} and \cite{Corvino}. In \cite{hepeterwylie} it was given some classification for $m$-quasi-Einstein metrics where the base has non-empty boundary. Moreover, they have proved a characterization for $m$-quasi-Einstein metric when the base is locally conformally flat. In addition, considering $m=\infty$ in equation (\ref{eqprinc}) we obtain the almost Ricci soliton equation, for more details see \cite{prrs} and \cite{br}. We also point out that, Catino \cite{catino} have proved that  around any regular point of $f$ a generalized $m$-quasi Einstein metric $\big(M^n,\,g,\,\nabla f,\,\lambda\big)$ with harmonic Weyl tensor and $W(\nabla f,\cdots,\nabla f) = 0$ is locally a warped product with $(n-1)$-dimensional Einstein fibers.

A generalized $m$-quasi-Einstein manifold $\big(M^n,\,g,\,\nabla f,\,\lambda\big)$  will be called \emph{trivial} if the potential function $f$ is constant. Otherwise, it will be called \emph{nontrivial}.

We observe that the triviality definition implies that $M^n$ is an Einstein manifold, but the converse is not true. Meanwhile, we shall show in Theorem \ref{thmsph} that when $(M^n,\,g,\,\nabla f,\,\lambda),\,n\ge 3,$ is Einstein, but not trivial, it will be isometric to a space form with a well defined potential $f$. Introducing the function $u=e^{-\frac{f}{m}}$ on $M$ we immediately have $\nabla u=-\frac{u}{m}\nabla f$, moreover the next relation, which can be found in \cite{csw}, is true
\begin{equation}
\label{conformal}
\nabla^2f -\frac{1}{m}df\otimes df=-\frac{m}{u}\nabla^2u.
\end{equation}In particular, $\nabla u $ is a conformal vector field, i.e. $\frac{1}{2}\mathcal{L}_{\nabla u}g=\rho\,g,$ for some smooth function $\rho$ defined on $M$, if and only if $M^n$ is an Einstein manifold. Hence, on a surface $M^2,\,\nabla u$ is always a conformal vector field.

Before to announce our main result we present a family of nontrivial examples on a space form. Let us start with a standard sphere $(\Bbb{S}^n,g_0),$ where $g_0$ is its canonical metric.
\begin{example}\label{exp1}On the standard unit sphere $(\Bbb{S}^n,g_0),\,n\ge2,$ we consider the following function
\begin{equation}
\label{potf}f=-m \ln \big(\tau-\frac{h_v}{n} \big),
\end{equation}where $\tau$ is a real parameter lying in $(1/n,+\infty)$ and $h_v$ is some height function with respect to a fixed unit vector $v\in \Bbb{S}^{n}\subset \Bbb{R}^{n+1}$, here we are considering   $\Bbb{S}^n$ as a hypersurface in  $\Bbb{R}^{n+1}$, and $h_v:\Bbb{S}^{n}\to \Bbb{R}$ is given by $h_v(x)=\langle x,v\rangle$. Taking into account that $\nabla^2h_v=-h_vg_0$ and $u=e^{-\frac{f}{m}}=\tau-\frac{h_v}{n}$, we deduce from (\ref{conformal}) that
\begin{equation}
\label{eqnexp1}
 \nabla^{2}f-\frac{1}{m}df\otimes df=-m\frac{\tau -u}{u}g_0.
 \end{equation}

Since the Ricci tensor of $(\Bbb{S}^n,g_0)$ is given by $Ric=(n-1)g_0,$ it is enough to consider $\lambda =(n-1)-m\frac{\tau -u}{u}$ in order to build a desired non trivial such structure on $(\Bbb{S}^n,g_0)$.
\end{example}

We now present a similar example as before on the Euclidean space $(\Bbb{R}^n,g_0),$ where $g_0$ is its canonical metric.
\begin{example}\label{exp2}On the Euclidean space $(\Bbb{R}^n,g_0),\,n\ge2,$ we consider the following function
\begin{equation}
\label{potf1}f=-m \ln \big(\tau+|x|^2 \big),
\end{equation}where $\tau$ is a positive real parameter and $|x|$ is the Euclidean norm. Taking into account that $\nabla^2|x|^2 =2g_0$ and $u=e^{-\frac{f}{m}}=\tau+|x|^2$, we deduce from (\ref{conformal}) that
\begin{equation}
\label{eqnexp1}
 \nabla^{2}f-\frac{1}{m}df\otimes df=-2\frac{m}{u}g_0.
 \end{equation}
  Since the Ricci tensor of $(\Bbb{R}^n,g_0)$ is flat, it is enough to consider $\lambda =-2\frac{m}{u}$ in order to obtain a desired non trivial structure on $(\Bbb{R}^n,g_0)$.
\end{example}
On the other hand, concerning to hyperbolic space we have the following.
\begin{example}\label{exp3}Regarding the hyperbolic space $\Bbb{H}^{n}(-1)\subset\Bbb{R}^{n,1}:\langle x,x\rangle_{0}=-1, x_{1}>0,$ where $\Bbb{R}^{n,1}$ is the Euclidean space $\Bbb{R}^{n+1}$ endowed with the inner product $\langle x,x\rangle_{0}=-x_{1}^{2}+x_{2}^{2}+\ldots+x_{n+1}^{2}$. We now follow the argument used on $\Bbb{S}^{n}.$ First, we  fixe a vector $v\in \Bbb{H}^{n}(-1)\subset\Bbb{R}^{n,1}$  and we consider a hight function $h_v:\Bbb{H}^{n}(-1)\to \Bbb{R}$ given by $h_v(x)=\langle x,v\rangle_{0}$. In this case, we have $\nabla^2 h_{v}=h_{v}g_{0}.$ Then, taking
\begin{equation}
\label{exhyp}u=e^{-\frac{f}{m}}=\tau+h_{v},\,\tau>-1
\end{equation}we have from (\ref{conformal})
\begin{equation}\label{eqhyperbolic}
\nabla^{2}f-\frac{1}{m}df\otimes df=-m\frac{u-\tau}{u}g_{0}.
\end{equation}
Reasoning as in the spherical case it is enough to consider $\lambda =-(n-1)-m\frac{\tau -u}{u}$ in order to build a non trivial such structure on $(\Bbb{H}^n,g_0)$.
\end{example}

Now we announce the main theorem.
\begin{theorem}
\label{thmsph}Let $\big(M^n,\,g,\,\nabla f,\,\lambda\big)$ be a non trivial generalized $m$-quasi-Einstein metric with $n\ge3.$ Suppose that either $(M^n,\,g)$ is an Einstein manifold or $\nabla u$ is a conformal vector field. Then one the following statements holds:
\begin{enumerate}
\item $M^n$ is isometric to a standard sphere $\Bbb{S}^{n}(r).$ In particular, $f$ is, up to constant, given by (\ref{potf}).
\item $M^{n}$ is isometric to a Euclidean space $\Bbb{R}^{n}$. In particular, $f$ is, up to change of coordinates, given by (\ref{potf1}).
\item $M^{n}$ is isometric to a hyperbolic space $\Bbb{H}^{n},$ provided $u$ has only one critical point. In particular, $f$ is, up to constant, given according to (\ref{exhyp}).
\end{enumerate}
\end{theorem}

As a consequence of this theorem we obtain the following corollary.
\begin{corollary}
\label{corthmsph}Let $\big(M^n,\,g,\,\nabla f,\,\lambda\big),$ \,$ n\ge3,$ be a compact non trivial generalized $m$-quasi-Einstein metric such that $\int_{M}Ric(\nabla u,\nabla u)d\mu\geq\frac{n-1}{n}\int_{M}(\Delta u)^{2}d\mu,$ where $d\mu $  stands for the Riemannian measure associated to $g$. Then $M^n$ is isometric to a standard sphere  $\Bbb{S}^{n}(r)$. Moreover, the potential $f$ is the same of identity (\ref{potf}).
\end{corollary}

Before to announce the next results we point out that they are generalizations of ones found in \cite{Petwy} and \cite{abr} for Ricci solitons, \cite{br} for almost Ricci solitons and \cite{csw} for quasi-Einstein metrics. First, we have the following theorem.
\begin{theorem}
\label{thm1}
Let $\big(M^n,\,g,\,\nabla f,\,\lambda\big)$ be a compact generalized $m$-quasi-Einstein metric. Then $M^n$ is trivial provided:
\begin{enumerate}
\item $\int_{M}Ric(\nabla f,\nabla f)d\mu\leq\frac{2}{m}\int_{M}|\nabla f|^{2}\Delta f d\mu-(n-2)\int_{M}\langle\nabla\lambda,\nabla f\rangle d\mu.$
\item $R\ge\lambda n$ or $R\le \lambda n$.
\end{enumerate}
\end{theorem}

Now, if $\big(M^n,\,g,\,\nabla f,\,\lambda\big)$ is a generalized $m$-quasi-Einstein metric  and $m$ is finite, we shall present conditions in order to obtain $\nabla f\equiv 0.$
\begin{theorem}
\label{thm2}
Let $\big(M^n,\,g,\,\nabla f,\,\lambda\big)$ be a complete generalized $m$-quasi-Einstein metric with $m$ finite. Then $\nabla f\equiv 0$, if one of the following conditions holds:
\begin{enumerate}
\item $M^n$ is non compact,
$n\lambda\geq R$ and $|\nabla f|\in \mathrm{L}^{1}(M^n)$.  In particular, $M^n$ is an Einstein manifold.
\item $(M^n,\,g)$ is Einstein and $\nabla f$ is a conformal vector filed.
\end{enumerate}
\end{theorem}

\section{Preliminaries}

In this section we shall present some preliminaries which will be useful for the establishment of the desired results.
The first one is a general lemma for a vector field $X\in \frak{X}(M^n)$ on a Riemannian manifold $M^n$.
\begin{lemma}
\label{lemflat}
Let $(M^{n},\,g)$ be a Riemannian manifold and $X\in \frak{X}(M^n)$. Then the following statements hold:
\begin{enumerate}
\item If\, $\big(X^{\flat}\otimes X^{\flat}\big)=\rho g$ for some smooth function $\rho:M\to \Bbb{R},$ then $\rho=|X|^2=0.$ In particular, the unique solution of the equation $df\otimes df=\rho g$ is $f$ constant.
\item If $M^n$ is compact and  $X$ is a conformal vector field, then $\int_{M}|X|^{2}div\,Xd\mu=0.$ In particular, if $X=\nabla f$ is a gradient conformal vector field, then $\int_{M}|\nabla f|^{2}\Delta fd\mu=0.$
\end{enumerate}
\end{lemma}
\begin{proof}Since $\big(X^{\flat}\otimes X^{\flat}\big)$ is a degenerate $(0,2)$ tensor the first statement is trivial.
Taking into account that $X$ is a conformal vector field  we have $\frac{1}{2}\mathcal{L}_{X}g=\rho\,g,$ where $\rho=\frac{1}{n}div\,X$. From which we obtain
\begin{equation}
\label{eqnconf2}
|X|^{2}div\,X=n\langle\nabla_{X}X,X\rangle.
\end{equation}
On the other hand, since $div\,(|X|^{2}X)=|X|^{2}div\,X+2\langle\nabla_{X}X,X\rangle$, one has

\begin{equation}
\label{eqnconf2}
div\,(|X|^{2}X)=\frac{n+2}{n}|X|^{2}div\,X,
\end{equation}which allows us to complete the proof of the lemma.
\end{proof}

The following formulae from \cite{Petwy} will be useful: on a Riemannian manifold $(M^n,\,g)$ we have
\begin{equation}
\label{eqbochner}
div\,(\mathcal{L}_{X}g)(X)=\frac{1}{2}\Delta|X|^{2}-|\nabla X|^{2}+Ric\,(X,X)+D_{X}div\,X,
\end{equation}

\begin{equation}
div\,(\mathcal{L}_{\nabla f}g)(Z)=2Ric\,(Z,\nabla f)+2D_{Z}div\,\nabla f,
\end{equation}
or on $(1,1)$-tensorial notation
\begin{equation}
\label{eqbochner1}
div\,\nabla^{2} f=Ric\,(\nabla f)+\nabla\Delta f
\end{equation}
and
\begin{equation}
\label{bochnerform}
\frac{1}{2}\Delta\,|\nabla f|^{2}=|\nabla ^2 f|^{2}+D_{\nabla f}div\nabla f+Ric(\nabla f,\nabla f).
\end{equation}

Taking into account that $div (\lambda I)(X)=\langle\nabla\lambda,X\rangle$, where $\lambda$ is a smooth function on $M^{n}$ and $X\in\mathfrak{X}(M)$,\,  equation (\ref{eqbochner}) allows us to deduce the following lemma.

\begin{lemma}
\label{lem1}Let $(M^{n},\,g,\,\nabla f,\,\lambda)$ be a generalized $m$-quasi-Einstein metric. Then we have
\begin{enumerate}
\item
  $\frac{1}{2}\Delta|\nabla f|^{2}=|\nabla^{2}f|^{2}-Ric(\nabla f, \nabla f)+\frac{2}{m}|\nabla f|^{2}\Delta f-(n-2)\langle\nabla\lambda,\nabla f\rangle.$
 \item\label{gradr} ${\frac{1}{2}\nabla R=\frac{m-1}{m}Ric (\nabla f)+\frac{1}{m}(R-(n-1)\lambda)\nabla f+(n-1)\nabla \lambda}.$
\item
\label{gradHamil}
${\nabla(R+|\nabla f|^{2}-2(n-1)\lambda)=2\lambda\nabla f+\frac{2}{m}\{\nabla_{\nabla f}\nabla f+(|\nabla f|^{2}-\Delta f)\nabla f\}}.$
  \end{enumerate}
\end{lemma}
\begin{proof}Since $Ric+\nabla^{2}f-\frac{1}{m}df \otimes df=\lambda g$ we use the second contracted Bianchi identity
\begin{equation}
\label{2bid}
\nabla R=2div\,Ric
\end{equation}as well as the next identity
\begin{equation}
\label{divdfodf}div\,(df\otimes df)=\Delta\,f\,\nabla f+\nabla_{\nabla f}\nabla f
\end{equation} and (\ref{eqbochner1}) to deduce

\begin{equation}
\label{eqngradrl1}
\nabla R+2Ric\,(\nabla f)+2\nabla\Delta f-\frac{2}{m}\Delta\,f\,\nabla f-\frac{2}{m}\nabla_{\nabla f}\nabla f=2\nabla\lambda.
\end{equation}
In particular one deduces
{\small\begin{equation}
\langle\nabla R,\nabla f\rangle+2Ric(\nabla f,\nabla f)+2\langle \nabla\Delta f,\nabla f\rangle-\frac{2}{m}\Delta\,f|\nabla f|^2-\frac{2}{m}\langle \nabla_{\nabla f}\nabla f,\nabla f\rangle=2\langle\nabla\lambda,\nabla f\rangle.
\end{equation}}
Next using (\ref{eqntrace2}) and (\ref{bochnerform}) jointly with the last identity we conclude

\begin{equation}
\label{eqitem1}
\frac{1}{2}\Delta|\nabla f|^{2}=|\nabla^{2} f|^{2}-Ric(\nabla f,\nabla f) +\frac{2}{m}|\nabla f|^{2}div\, \nabla f-(n-2)\langle\nabla\lambda,\nabla f\rangle,
\end{equation}
which finishes the first statement of the lemma.
On the other hand, substituting $\Delta f=-R+\lambda n+\frac{1}{m}|\nabla f|^{2}$ and remembering that $\nabla |\nabla f|^2=2\nabla_{\nabla f}{\nabla f}$ we use once more (\ref{eqngradrl1}) to write
\begin{eqnarray*}
\frac{1}{2}\nabla R&=&-Ric (\nabla f)-\nabla (-R+\lambda n+\frac{1}{m}|\nabla f|^{2})+\frac{1}{m}\Delta f \nabla f+\frac{1}{m}\nabla_{\nabla f}\nabla f+\nabla \lambda\\
&=&-Ric (\nabla f)+\nabla R-\frac{1}{m}\nabla_{\nabla f}\nabla f+\frac{1}{m}\Delta f \nabla f-(n-1)\nabla \lambda.
\end{eqnarray*}Of which we deduce
\begin{equation}
\label{equatl1}
\frac{1}{2}\nabla R=Ric(\nabla f)-\frac{1}{m}\Delta f \nabla f+\frac{1}{m}\nabla_{\nabla f}\nabla f+(n-1)\nabla \lambda.
\end{equation}

We now use the fundamental equation to write
\begin{equation}
\label{equatl2}
\nabla_{\nabla f}\nabla f=\lambda\nabla f+\frac{1}{m}|\nabla f|^{2}\nabla f-Ric(\nabla f).
\end{equation}
In particular, combining (\ref{equatl1}) and  (\ref{equatl2}) we obtain

\begin{eqnarray*}
\frac{1}{2}\nabla R &=& \frac{m-1}{m}Ric(\nabla f)+\frac{1}{m}\Big(\lambda+\frac{1}{m}|\nabla f|^{2}-\Delta f\Big) \nabla f+(n-1)\nabla \lambda\\
&=&\frac{m-1}{m}Ric(\nabla f)+\frac{1}{m}(R-(n-1)\lambda)\nabla f+(n-1)\nabla \lambda,
\end{eqnarray*}which gives the second assertion.

Finally, noticing that $\frac{1}{2}\nabla R+\frac{1}{2}\nabla|\nabla f|^{2}=\frac{1}{2}\nabla R+\nabla_{\nabla f}{\nabla f}$ we use the last equation and (\ref{equatl2}) to write
\begin{eqnarray*}
\frac{1}{2}\nabla R+\frac{1}{2}\nabla|\nabla f|^{2}
&=&\frac{m-1}{m}Ric(\nabla f)+\frac{1}{m}(R-(n-1)\lambda)\nabla f+(n-1)\nabla\lambda\\&+&\lambda\nabla f+\frac{1}{m}|\nabla f|^{2}\nabla f-Ric(\nabla f)\\
\end{eqnarray*}
Thus, using equation (\ref{eqntrace}) once more, we achieve
{\small\begin{eqnarray*}
\nabla(R+|\nabla f|^{2}-2(n-1)\lambda)-2\lambda\nabla f&=&\frac{2}{m}\{(|\nabla f|^{2}+R-(n-1)\lambda)\nabla f-Ric(\nabla f)\}\\&=&\frac{2}{m}\{(|\nabla f|^{2}+R-n\lambda+\lambda)\nabla f-Ric(\nabla f)\}
\\&=&\frac{2}{m}\{(|\nabla f|^{2}+\frac{1}{m}|\nabla f|^{2}-\Delta f+\lambda)\nabla f-Ric(\nabla f)\}\\
%&=&\frac{2}{m}\{(|\nabla f|^{2}-\Delta f)\nabla f+\frac{1}{m}|\nabla f|^{2}\nabla f+\lambda\nabla f -Ric(\nabla f)\}\\
&=&\frac{2}{m}\{\nabla_{\nabla f}\nabla f+(|\nabla f|^{2}-\Delta f)\nabla f\},\\
\end{eqnarray*}}which concludes the proof of the lemma.
\end{proof}

It is convenient to point out that  for $m=\infty$ and $\lambda$ constant, assertion (\ref{gradHamil}) of the last lemma is a generalization of the classical Hamilton equation \cite{hamilton} for a gradient Ricci soliton: $\mathrm{R+|\nabla f|^{2}-2\lambda f=C}$, where $C$ is constant, as well as for the following relation: $\mathrm{\nabla(R+|\nabla f|^{2}-2(n-1)\lambda)=2\lambda\nabla f}$, that was proved in \cite{br} for an almost Ricci soliton.
Choosing $Z\in\mathfrak{X}(M),$ we deduce from the first assertion of Lemma \ref{lem1} the following identity
\begin{equation}
\label{equatle}
\frac{1}{2}\langle\nabla R,Z\rangle=\frac{m-1}{m}Ric(\nabla f,Z)+\frac{1}{m}(R-(n-1)\lambda)\langle\nabla f,Z\rangle+(n-1)\langle\nabla\lambda,Z\rangle.
\end{equation}
We now present the main result of this section. Taking in account that $u=e^{-\frac{f}{m}}$ we have the following lemma.
\begin{lemma}
\label{lemM}Let $(M^{n},\,g,\,\nabla f,\,\lambda),\,n\ge3,$ be a generalized $m$-quasi-Einstein metric. If, in addition $M^n$ is Einstein, then we have
\begin{equation}\nabla^{2} u=\big(-\frac{R}{n(n-1)}u+\frac{c}{m}\big)g,
\end{equation}where $c$ is constant.
\end{lemma}

\begin{proof}Since $M^n$ is Einstein and $n\ge3$ we have $Ric=\frac{R}{n}g$ with $R$ constant. In particular, it follows from (\ref{conformal}) that
\begin{equation}
\label{eqnd2uM}
\nabla^2u=\frac{1}{m}\big( \frac{R}{n}u-\lambda u\big)g.
\end{equation}Whence, using (\ref{eqbochner1}) we deduce
\begin{equation}
\label{eqdivd2u}
Ric\,(\nabla u)+\nabla\Delta u=\frac{1}{m}\nabla\big( \frac{R}{n}u-\lambda u\big).
\end{equation}Therefore we infer

\begin{equation}
\label{eqdivd2u1}
\frac{R}{n}\nabla u+\nabla\Delta u=\frac{R}{nm}\nabla u-\frac{1}{m}\nabla( \lambda u).
\end{equation}

On the other hand, in accordance with (\ref{eqprinc}) and (\ref{conformal}) we deduce

\begin{equation}
\label{eqnlapu1}
\Delta u=\frac{R}{m}u-\frac{n}{m}\lambda u.
\end{equation}

We now  compare (\ref{eqdivd2u1}) and (\ref{eqnlapu1}) to obtain
\begin{equation}
\label{eqngradlambdau}
\nabla (\lambda u)=R\frac{(m+n-1)}{n(n-1)}\nabla u.
\end{equation}Therefore we deduce $\lambda u=R\frac{(m+n-1)}{n(n-1)}u-c,$  where $c$ is constant. Next we use this value of $\lambda u$ in (\ref{eqnd2uM}) to complete the proof of the lemma.

\end{proof}

\section{Proofs of the main results}

\subsection{Proof of Theorem \ref{thmsph}}
\begin{proof}
First of all, we notice that (\ref{conformal}) gives that $M^n$ is Einstein if and only if  $\nabla u$ is a conformal vector field. Since $f$ is not constant and we are supposing that $\nabla u$ is a non trivial conformal vector field, which enables us to write $\frac{1}{2}\mathcal{L}_{\nabla u}g=\nabla^2u=\frac{\Delta\,u}{n}g$, we deduce that $M^n$ is Einstein. Moreover, using (\ref{eqprinc}) and (\ref{conformal}) we deduce
\begin{equation*}
Ric\,=\big(\lambda+m\frac{\Delta\, u}{nu}\big)g.
\end{equation*}Since $n\ge3,$ we have from Schur's Lemma that $R=n\lambda+m\frac{\Delta\,u}{u}$ is constant.

On the other hand, from Lemma \ref{lemM}  we have $$\nabla^{2} u=\big(-\frac{R}{n(n-1)}u+\frac{c}{m}\big)g$$ where $c$ is constant. Therefore, we are in position to apply Theorem $2$ due to Tashiro \cite{tash} to deduce that $M^n$ is a space form.

If $R$ is positive,  we may assume that $M^{n}$ is isometric to a unit standard sphere $\Bbb{S}^{n}.$ Since $R=n(n-1)$ we deduce from Lemma \ref{lemM} that  $\Delta u+nu=kn,$ where $k$ is constant. Then, up to constant, $u$ is a first eigenfunction of the Laplacian of $\Bbb{S}^{n}.$ Therefore, we have $u=h_{v}(x)=\langle x,v\rangle+k$, where $v$ is a linear combination of unit vectors  in $\Bbb{R}^{n+1}$. Hence, $f$ is, up to constant, given by (\ref{potf}).

Next, if $R=0$ we have from (\ref{eqngradlambdau}) that $c$ is not zero. In this case $M^{n}$ is isometric to a Euclidean space $\Bbb{R}^{n}.$ Using once more  Lemma \ref{lemM} we obtain $\Delta u =k,$ where $k$ is constant. Since $u$ must be positive, up change of coordinates, we deduce that $u(x)=|x|^2+\tau$, with $\tau>0.$

Finally, if $R<0,$ it follows from Theorem $2$ of \cite{tash} that $M^n$ is isometric to a hyperbolic space, since we have only one critical point for $u.$ Now let us suppose that $M^n$ is isometric to $\Bbb{H}^{n}(-1).$ We can use the same argument due to Tashiro \cite{tash} to conclude that, up to constant, $u=h_{v}+\tau,\,\tau>-1$, with $v\in \Bbb{H}^{n}(-1)$, since in this case $\langle x,v\rangle_{0}=-\cosh \eta(x,v)$, where $\eta(x,v)$ is the time-like angle between $x$ and $v,$ which is exactly the geodesic distance between them. Therefore, we complete the proof of the theorem.

\end{proof}

\subsection{Proof of Corollary \ref{corthmsph}}
\begin{proof} On integrating Bochner's formula we obtain
\begin{equation}
\label{eqD2u}
\int_{M}|\nabla^{2}u-\frac{\Delta u}{n}g|^{2}d\mu=\frac{n-1}{n}\int_{M}(\Delta u)^{2}d\mu-\int_{M}Ric(\nabla u,\nabla u)d\mu.
\end{equation}

In particular, from our assumption we conclude that
\begin{equation}\label{eqintd2u}
\int_{M}|\nabla^{2}u-\frac{\Delta u}{n}g|^{2}d\mu=0.
\end{equation}

Whence, we deduce that $\nabla u$  is a non trivial conformal vector field. Then, for $n\ge3,$ we can apply Theorem \ref{thmsph} to conclude the proof of the corollary.
\end{proof}

\subsection{Proof of Theorem \ref{thm1}}
\begin{proof}First we integrate the identity derived in Lemma \ref{lem1} and Stokes' formula to infer

\begin{equation}\label{eqnthm11}
\int_{M}|\nabla^2 f|^{2}d\mu=\int_{M}Ric\,(\nabla f,\nabla f)d\mu-\frac{2}{m}\int_{M}|\nabla f|^{2}\Delta f d\mu+(n-2)\int_{M}\langle\nabla\lambda,\nabla f\rangle d\mu.
\end{equation}
On the other hand, since we are assuming that the right hand of above identity is less than or equal to zero, we obtain $\nabla^{2}f=0$. Therefore, $\Delta\,f=0,$ which implies by Hopf's theorem that $f$ is constant and we finish the establishment  of the first assertion.

Proceeding one notices that for $m=\infty$, using equation (\ref{eqntrace}) the result follows. On the other hand, for $m$ finite, considering
once more the  auxiliary function $u=e^{-\frac{f}{m}},$ as we already saw
$\Delta u=\frac{u}{m}(R-\lambda n).$ Since $M^{n}$ is compact, $u>0$  and $(\mathrm{R-n\lambda})\geq0\,(\leq0)$, we can use once more Hopf's theorem to deduce that $u$ is constant and so is $f$. From which we complete the proof of the theorem.
\end{proof}

\subsection{Proof of Theorem \ref{thm2}}
\begin{proof}Taking into account identity (\ref{eqntrace}) we obtain
\begin{equation}
\label{eqproteo2}
mdiv\, \nabla f=|\nabla f|^{2}+m(n\lambda-R).
\end{equation}
By one hand $mdiv\, \nabla f\geq0$, since $(n\lambda-R)\geq0$. On the other hand, if  $|\nabla f|\in L^{1}(M^n)$, we may invoke Proposition $1$ in \cite{caminha}, which is a generalization of  a result due to Yau \cite{yau} for subharmonic functions, to derive that $div\, \nabla f=0.$ Next, we may use equation (\ref{eqproteo2}) to conclude that $\nabla f\equiv0,$ as well as $n\lambda=R.$ Therefore, $f$ is constant and $M^n$ is an Einstein manifold, which gives the first assertion.
Now let us suppose that $(M^n,\,g)$ is an Einstein manifold, in particular a surface has this propriety. If $\nabla f$ is a conformal vector field with conformal factor $\rho,$ here we can have a Killing vector field, then $\nabla^{2}f=\rho g$, where $\rho=\frac{1}{n}\rm{div}\,\nabla f$. Since $Ric\,=\frac{R}{n}g$ we deduce from equation (\ref{eqntrfr}) that
\begin{equation}
\label{eqthm6}
\frac{1}{m}(df\otimes df)=|\nabla f|^2 g.
\end{equation}
But, using that $m$ is finite, we can apply Lemma \ref{lemflat} to conclude that $\nabla f\equiv 0,$ which completes the proof of the theorem.
\end{proof}
\section{Integral formulae for generalized $m$-quasi-Einstein metrics}
In this section we shall introduce some integral formulae for a compact generalized $m$-quasi-Einstein metric. Before, we present the next result which is a natural extension of one obtained for an almost Ricci soliton in \cite{br}, as well as a similar one in \cite{prrs}.

\begin{lemma}
\label{lem4}
Let $\big(M^n,\,g,\,\nabla f,\,\lambda\big)$ be a generalized $m$-quasi-Einstein metric. Then we have
\begin{eqnarray*}
\frac{1}{2}\Delta R&=&-|\nabla^{2}f-\frac{\Delta f}{n}g|^{2}-\Big\{\frac{m+n}{nm}\Big\}(\Delta f)^{2}-\frac{n}{2}\langle\nabla f,\nabla\lambda\rangle
+\langle\nabla f,\nabla R\rangle\\&&+\Big\{\frac{m-2}{2m}\Big\}\langle\nabla f,\nabla \Delta f\rangle+\frac{1}{m}div\,\big(\nabla_{\nabla f}\nabla f\big)+(n-1)\Delta\lambda+\lambda\Delta f.\\
\end{eqnarray*}
\end{lemma}

\begin{proof}Initially by using assertion (\ref{gradHamil}) of Lemma \ref{lem1} to  compute the divergence of $\nabla R$ we obtain
\begin{eqnarray*}
\Delta R+\Delta|\nabla f|^{2}-2(n-1)\Delta\lambda&=&2div\,(\lambda\nabla f)+\frac{2}{m}\Big\{\langle\nabla(|\nabla f|^{2}-\Delta f),\nabla f\rangle\\&+&(|\nabla f|^{2}-\Delta f)\Delta f+div\,(\nabla_{\nabla f}\nabla f)\Big\}.
\end{eqnarray*}
We now use $|\nabla^{2}f-\frac{\Delta f}{n}g|^{2}=|\nabla^{2}f|^{2}-\frac{1}{n}(\Delta f)^{2}$ with Bochner's formula to write
{\small\begin{eqnarray*}
\frac{1}{2}\Delta R&=&-Ric\,(\nabla f,\nabla f)-|\nabla^{2}f-\frac{\Delta f}{n}g|^{2}-\frac{1}{n}(\Delta f)^{2}-\langle\nabla\Delta f,\nabla f\rangle\\&+&(n-1)\Delta\lambda+div\,(\lambda\nabla f)+\frac{2}{m}\langle\nabla_{\nabla f}\nabla f,\nabla f\rangle\\&+&\frac{1}{m}\Big\{(|\nabla f|^{2}-\Delta f)\Delta f-\langle\nabla\Delta f,\nabla f\rangle+div\,(\nabla_{\nabla f}\nabla f)\Big\}.
\end{eqnarray*}}
Next, we invoke equation (\ref{eqntrace}) to write $\langle\nabla\Delta f,\nabla f\rangle= \langle\nabla\big(n\lambda +\frac{1}{m}|\nabla f|^2-R\big),\nabla f\rangle$. Then the last relation becomes
\begin{eqnarray*}
\frac{1}{2}\Delta R&=&-Ric\,(\nabla f,\nabla f)-|\nabla^{2}f-\frac{\Delta f}{n}g|^{2}-\frac{m+n}{nm}(\Delta f)^{2} +(n-1)\Delta\lambda\\&-&\langle\nabla(\frac{1}{m}|\nabla f|^{2}-R+\lambda n),\nabla f\rangle+\frac{2}{m}\langle\nabla_{\nabla f}\nabla f,\nabla f\rangle+div\,(\lambda\nabla f)\\&+&\frac{1}{m}\Big\{|\nabla f|^{2}\Delta f-\langle\nabla\Delta f,\nabla f\rangle+div\,(\nabla_{\nabla f}\nabla f)\Big\}\\
&=&-\big(Ric\,(\nabla f,\nabla f)+(n-1)\langle\nabla\lambda,\nabla f\rangle\big)-|\nabla^{2}f-\frac{\Delta f}{n}g|^{2}-\frac{m+n}{nm}(\Delta f)^{2}\\&+&(n-1)\Delta\lambda+\lambda\Delta f+\langle\nabla R,\nabla f\rangle\\&+& \frac{1}{m}\Big\{|\nabla f|^{2}\Delta f-\langle\nabla\Delta f,\nabla f\rangle+div\,(\nabla_{\nabla f}\nabla f)\Big\}.
\end{eqnarray*}
On the other hand, using (\ref{equatle}) we can write
\begin{equation}
\label{eqric1}
 Ric(\nabla f,\nabla f)+(n-1)\langle\nabla\lambda,\nabla f\rangle=\frac{1}{2}\langle\nabla R,\nabla f\rangle+\frac{1}{m}Ric(\nabla f,\nabla f)
-\frac{1}{m}(R-(n-1)\lambda)|\nabla f|^2.
\end{equation}

Therefore, we compare the last two equations to obtain
\begin{eqnarray*}
\frac{1}{2}\Delta R
&=&\frac{1}{2}\langle\nabla R,\nabla f\rangle-|\nabla^{2}f-\frac{\Delta f}{n}g|^{2}-\frac{m+n}{nm}(\Delta f)^{2}+(n-1)\Delta\lambda+\lambda\Delta f\\&+& \frac{1}{m}\Big\{-Ric(\nabla f,\nabla f)+\big(\Delta f+R-n\lambda \big)|\nabla f|^{2}+\lambda|\nabla f|^{2}\Big\}\\
\\&+& \frac{1}{m}\Big\{-\langle\nabla\Delta f,\nabla f\rangle+div\,(\nabla_{\nabla f}\nabla f)\Big\}\\
&=&\frac{1}{2}\langle\nabla R,\nabla f\rangle-|\nabla^{2}f-\frac{\Delta f}{n}g|^{2}-\frac{m+n}{nm}(\Delta f)^{2}+(n-1)\Delta\lambda+\lambda\Delta f\\&+& \frac{1}{m}\Big\{\langle\nabla_{\nabla f}\nabla f,\nabla f\rangle-\langle\nabla\Delta f,\nabla f\rangle+div\,(\nabla_{\nabla f}\nabla f)\Big\}\\
&=&\frac{1}{2}\langle\nabla R,\nabla f\rangle-|\nabla^{2}f-\frac{\Delta f}{n}g|^{2}-\frac{m+n}{nm}(\Delta f)^{2}+(n-1)\Delta\lambda+\lambda\Delta f\\&+& \frac{1}{2}\langle\nabla R,\nabla f\rangle+\frac{1}{2}\langle\nabla f,\nabla \Delta f\rangle-\frac{n}{2}\langle\nabla \lambda,\nabla f\rangle\\&-&\frac{1}{m}\langle\nabla\Delta f,\nabla f\rangle+\frac{1}{m}div\,(\nabla_{\nabla f}\nabla f).
\end{eqnarray*}

We now group terms to arrive at the desired result, hence we complete the proof of the lemma.

\end{proof}

As a consequence of this lemma we obtain the following integral formulae.

\begin{theorem}
\label{thm5}
Let $\big(M^{n},\,g,\,\nabla f,\,\lambda \big)$ be a compact orientable generalized $m$-quasi-Einstein metric. Then we have.

\begin{enumerate}
\item $\int_{M}|\nabla^{2}f-\frac{\Delta f}{n}g|^{2}d\mu+\frac{n+2}{2n}\int_{M}(\Delta f)^{2}d\mu=\int_{M}\langle\nabla f,\nabla R\rangle d\mu-\frac{n+2}{2}\int_{M}\langle\nabla f,\nabla\lambda\rangle d\mu.$
\item $\int_{M}\big(Ric\,(\nabla f,\nabla f)+\langle\nabla f,\nabla R\rangle\big) d\mu=\frac{3}{2}\int_{M}(\Delta f)^{2}d\mu+\frac{n+2}{2}\int_{M}\langle\nabla f,\nabla\lambda\rangle d\mu.$

\item $M^n$ is trivial, provided $\int_{M}\langle\nabla R,\nabla f\rangle d\mu\leq\frac{n+2}{2}\int_{M}\langle\nabla f,\nabla\lambda\rangle d\mu.$

\item $\int_{M}|\nabla^{2}f-\frac{\Delta f}{n}g|^{2}d\mu=\frac{n-2}{2n}\int_{M}\langle\nabla f,\nabla R\rangle d\mu-\frac{n+2}{2nm}\int_{M}|\nabla f|^{2}\Delta f d\mu.$

\end{enumerate}
\end{theorem}
\begin{proof} Since $M^n$ is compact we use Lemma \ref{lem4} and Stokes' formula to infer
\begin{eqnarray*}
\int_{M}|\nabla^{2}f-\frac{\Delta f}{n}g|^{2}d\mu&=&-\Big(\frac{m+n}{nm}\Big)\int_{M}(\Delta f)^{2}d\mu-\Big(\frac{m-2}{2m}\Big)\int_{M}(\Delta f)^{2}d\mu\\&-&\frac{n}{2}\int_{M}\langle\nabla\lambda,\nabla f\rangle d\mu-\int_{M}\langle\nabla\lambda,\nabla f\rangle d\mu+\int_{M}\langle\nabla f,\nabla R\rangle d\mu.
\end{eqnarray*}
Therefore, we obtain
\begin{equation}
\label{intphif2}
\int_{M}\Big(|\nabla^{2}f-\frac{\Delta f}{n}g|^{2}+\frac{n+2}{2n}(\Delta f)^{2}\Big)d\mu=\int_{M}\langle\nabla f,\nabla R\rangle d\mu-\frac{n+2}{2}\int_{M}\langle\nabla f,\nabla\lambda\rangle d\mu,
\end{equation}which gives the first statement.

Next, we integrate Bochner's formula to get
\begin{equation}
\label{intboch}
\int_{M}Ric\,(\nabla f,\nabla f)d\mu+\int_{M}|\nabla^{2}f|^{2}d\mu+\int_{M}\langle\nabla f,\nabla\Delta f\rangle d\mu=0.
\end{equation}
Since $\int_{M}|\nabla^{2}f-\frac{\Delta f}{n}g|^{2}d\mu=\int_{M}|\nabla^{2}f|^{2}d\mu-\frac{1}{n}\int_{M}(\Delta f)^{2}d\mu$ we use Stokes' formula once more to deduce
\begin{equation}
\label{eqco1-2}
\int_{M}Ric\,(\nabla f,\nabla f)d\mu+\int_{M}|\nabla^{2}f-\frac{\Delta f}{n}g|^{2}d\mu=\frac{n-1}{n}\int_{M}(\Delta f)^{2}d\mu.
\end{equation}

Now, comparing (\ref{intphif2}) with (\ref{eqco1-2}) we obtain
$$\int_{M}\big(Ric\,(\nabla f,\nabla f)+\langle\nabla f,\nabla R\rangle\big)d\mu=\frac{3}{2}\int_{M}(\Delta f)^{2}d\mu+\frac{n+2}{2}\int_{M}\langle\nabla f,\nabla\lambda\rangle d\mu,$$ that was to be proved.

On the other hand, if $\int_{M}\langle\nabla R,\nabla f\rangle d\mu\leq\frac{n+2}{2}\int_{M}\langle\nabla f,\nabla\lambda\rangle d\mu,$ in particular this occurs if $R$ and $\lambda$ are both constant, we deduce from the first assertion
\begin{equation}
\label{eqncor1}
\int_{M}|\nabla^{2}f-\frac{\Delta f}{n}g|^{2}d\mu+\frac{n+2}{2n}\int_{M}(\Delta f)^{2}d\mu=0,
 \end{equation}which implies that $f$ must be constant, so $M^n$ is trivial.

Finally, from (\ref{eqntrace}) we can write $\int_{M}\langle\nabla f,\nabla\lambda\rangle d\mu=\frac{1}{n}\int_{M}\langle\nabla f,\nabla(R+\Delta f-\frac{1}{m}|\nabla f|^{2})\rangle d\mu.$ Hence, by using equation (\ref{intphif2}) we infer
\begin{eqnarray*}
\int_{M}\Big(|\nabla^{2}f-\frac{\Delta f}{n}g|^{2}+\frac{n+2}{2n}(\Delta f)^{2}\Big)d\mu&=&\frac{n-2}{2n}\int_{M}\langle\nabla f,\nabla R\rangle d\mu+\frac{n+2}{2n}\int_{M}(\Delta f)^{2}d\mu\nonumber\\&+&\frac{n+2}{2nm}\int_{M}\langle\nabla f,\nabla|\nabla f|^{2}\rangle d\mu.
\end{eqnarray*}
Therefore, after cancelations and using Stokes' formula, we deduce
\begin{equation*}
\int_{M}|\nabla^{2}f-\frac{\Delta f}{n}g|^{2}d\mu=\frac{n-2}{2n}\int_{M}\langle\nabla f,\nabla R\rangle d\mu-\frac{n+2}{2nm}\int_{M}|\nabla f|^{2}\Delta fd\mu,
\end{equation*}which completes the proof of the theorem.
\end{proof}

Now we remember that for a conformal vector field $X$ on a compact Riemannian manifold $M^n$ we have $\int_{M}\mathcal{L}_{X}Rd\mu=\int_{M}\langle X,\nabla R\rangle d\mu=0$, see e.g. \cite{bezin}. On the other hand, from Lemma \ref{lemflat} we also have $\int_{M}|X|^2 div X d\mu=0$. Hence , using the last item  of the above theorem we deduce that the converse of those two results are true for a gradient vector field. More exactly, we have the following corollary.

\begin{corollary}Let $\big(M^{n},\,g,\,\nabla f,\,\lambda \big)$ be a compact orientable generalized $m$-quasi-Einstein metric with $m$ finite. Then we have.
\begin{enumerate}
\item If $n\ge3, \,\int_{M}\langle\nabla f,\nabla R\rangle d\mu=0$ and $\int_{M}|\nabla f|^2 \Delta f d\mu=0$, then $\nabla f$ is a conformal vector field.
\item If $n=2$ and $\int_{M}|\nabla f|^2 \Delta f d\mu=0$, then $f$ is constant.
\end{enumerate}
\end{corollary}
\begin{proof}For the first statement we use the last item of Theorem \ref{thm5} to deduce $\nabla^{2}f=\frac{\Delta f}{n}g$, which gives that $\nabla f$ is conformal. Next, we notice that for $n=2,$ it is enough to suppose $\int_{M}|\nabla f|^2 \Delta f d\mu=0$ to conclude that $\nabla f$ is conformal. But, using Theorem \ref{thm2} we conclude that $f$ is constant, which completes the proof of the corollary.

\end{proof}

\end{document}